\documentclass[twoside,11pt,reqno]{amsart}
\usepackage{amsmath,amssymb,amscd,mathrsfs,amscd}
\usepackage{graphics,verbatim, hyperref}
\usepackage{amsthm}
\usepackage{graphicx}
\usepackage{amsfonts}
\usepackage{amsopn}
\usepackage{amsmath, amsfonts, amssymb}
\usepackage{tikz}
\usepackage{url}
\usepackage{mathabx}

\newtheorem{thm}{Theorem}[subsection]
\newtheorem{dfn}[thm]{Definition}

\newtheorem{lem}[thm]{Lemma}
\newtheorem{prop}[thm]{Proposition}
\newtheorem{rk}[thm]{Remark}

\newcommand{\N}{\mathbb{N}}
\newcommand{\Z}{\mathbb{Z}}

\newcommand{\C}{\mathbb{C}}

\newcommand{\p}{\mathbf{P}.}
\newcommand{\rv}{\mathbf{V}.}

\newcommand{\osp}{\mathfrak{osp}}
\newcommand{\s}{\mathfrak{sl}}
\newcommand{\so}{\mathfrak{so}}
\newcommand{\Ext}{\operatorname{Ext}}

\newcommand{\atyp}{\operatorname{atyp}}
\newcommand{\fg}{\mathfrak{g}}
\newcommand{\fh}{\mathfrak{h}}

\newcommand{\ff}{\mathfrak{f}}
\newcommand{\gl}{\mathfrak{gl}}
\newcommand{\F}{\mathcal{F}}
\newcommand{\fv}{\mathcal{V}}
\newcommand{\X}{\mathcal{X}}
\newcommand{\0}{\bar 0}
\newcommand{\1}{\bar 1}

\def \dd {D(2,1;\al)}

\def \dd {D(2,1;\alpha)}

\let \l=\lambda

\let \a=\alpha

\let \d=\delta

\let \l=\lambda

\let \eps=\varepsilon

\numberwithin{equation}{subsection}

\begin{document}

\title{Complexity of simple modules over the Lie superalgebra $\osp(k|2)$}
\author{Houssein El Turkey}

\address{Department of Mathematics and Physics, University of New Haven, 300 Boston Post Road, West Haven, CT 06516}
\email{helturkey@newhaven.edu}
\date{\today}
\maketitle
\begin{abstract}
The complexity of a module is the rate of growth of the minimal projective resolution of the module while the $z$-complexity is the rate of growth of the number of indecomposable summands at each step in the resolution. Let $\mathfrak{g}=\mathfrak{osp}(k|2)$ ($k>2$) be the type II orthosymplectic Lie superalgebra of types $B$ or $D$.  In this paper, we compute the complexity and the $z$-complexity of the simple finite-dimensional $\fg$-supermodules. We then give these complexities certain geometric interpretations using support and associated varieties.  
\end{abstract}


\section{Introduction}

The complexity of a module over a finite group was first introduced by Alperin \cite{Alp}. The complexity of a module (cf. Subsection~\ref{ss:comp}) is the rate of growth of a minimal projective resolution of the module. In \cite{Car},  Carlson introduced support varieties of modules over group algebras which were then used to give a geometric interpretation of the complexity. In particular, the dimension of the support variety of the module is equal to the complexity. 

This work was extended to Lie superalgebras in \cite{BKN1} where the authors computed the complexity of the simple and Kac modules over the general linear Lie superalgebra $\fg=\mathfrak{gl}(m|n)$ of type $A$. Their work was carried in the category of finite-dimensional $\fg$-supermodules which are completely reducible over the even part $\fg_{\0}$. Let $\F$ denote this category (see \cite{BKN2, BKN3}). The authors in \cite{BKN2} showed that, for any basic classical Lie superalgebra $\fg$, $\F$ has enough projective modules and it satisfies: $(i)$ it is a self-injective category and $(ii)$ every module in this category admits a projective resolution which has a polynomial rate of growth. For a module $M\in \F$, let $c_{\F}(M)$ denote the complexity of $M$. It was shown in \cite[Theorem~2.5.1]{BKN2} that the complexity is always finite in classical Lie superalgebras. It is also important to mention that complexity detects projectivity. In particular, by \cite[Corollary~2.7.1]{BKN2}, $c_{\F}(M)=0$ if and only if $M$ is projective. We can look at complexity as a tool of measuring how far the module is from being projective. 

In addition to computing the complexity of simple and Kac modules over $\gl(m|n)$ in \cite{BKN1}, the authors interpreted their computations geometrically using two types of varieties. Explicitly, if $\X_M$ denotes the associated variety defined by Duflo and Serganova \cite {DS}, and $\fv_{(\fg,\fg_{\0})}(M)$ is the support variety as defined in \cite{BKN3}, then:
\begin{equation} \label{compinterpret}
c_{\F}(X(\lambda))=\dim \X_{X(\lambda)}+\dim \fv_{(\fg,\fg_{\0})}(X(\lambda)),
\end{equation}
where $X(\lambda)$ is a Kac or a simple $\mathfrak{gl}(m|n)$-module.

In \cite{H}, the author computed the complexity of the simple and the Kac modules over the orthosymplectic Lie superalgebra $\osp(2|2n)$ of type $C$. These computations were interpreted geometrically as in \eqref{compinterpret}.
It is important to note that both types $A$ and $C$ are type $I$ Lie superalgebras and hence similar results were expected. However, it was shown in \cite{H} that simple modules satisfy \eqref{compinterpret} for the type $II$ Lie superalgebras: $\osp(3|2)$, and the three exceptional ones $\dd$, $G(3)$, and $F(4)$.

In this paper we compute the complexity of the simple modules over the orthosymplectic Lie superalgebras $\osp(k|2)$ ($k>2$) of types $B$ and $D$. We show in Theorem~\ref{t:comp} that the complexity of the atypical (cf. Subsection~\ref{ss:atyp}) simple modules is $k+1$. Then we verify that \eqref{compinterpret} holds in both types.

Moreover, we compute the $z$-complexity of simple modules over  $\osp(k|2)$ ($k>2$). The $z$-complexity of modules was introduced in \cite[Section~9]{BKN1} and will be denoted by $z_{\F}(-)$. In \cite{BKN1} the authors computed this categorical invariant for the simple and the Kac modules over $\gl(m|n)$ and then used a detecting subsuperalgebra $\ff$ to interpret their computations geometrically. The same was done for the Lie superalgebras $\osp(2|2n)$, $\osp(3|2)$, $\dd$, $G(3)$, and $F(4)$ in \cite{H}. We carry these computations over $\osp(k|2)$ and conclude in Theorem~\ref{ospgeometric2} that if $L(\lambda)$ is a simple module, we have:
\begin{equation}\label{zcompinterpret}
z_{\F}(L(\lambda))=\dim \fv_{(\ff,\ff_{\0})}(L(\lambda)).
\end{equation}

The organization of the paper goes as follows. In Section~\ref{s:prel}, we give brief preliminaries for classical Lie superalgebras and their representations. We also provide brief definitions for the complexity, support and associated variety, and $z$-complexity of modules. In Section~\ref{s:computecomplexity}, we compute the complexity of the simple modules over $\osp(k|2)$ for $k>2$ and show that \eqref{compinterpret} holds. In Section~\ref{s:computeZcomplexity}, we compute the $z$-complexity of the same family of modules and show \eqref{zcompinterpret} holds. 

\section{Preliminaries}\label{s:prel}
In this section we present some preliminary material on Lie superalgebras. We also give brief definitions for the complexity, $z$-complexity, support and associated varieties. Then we shift our focus to the Lie superalgebra $\osp(k|2)$ and we provide a set of technical tools that will be used in our computations.

\subsection{Lie superalgebras and representations}
We will mainly use the same notations and conventions from \cite{H} and \cite{SZ}. We will work over the complex numbers $\C$ throughout this paper.

 A \emph{Lie superalgebra} $\fg$ is $\Z_2$-graded vector space $\fg=\fg_{\0}\oplus\fg_{\1}$ with a bracket operation $[\,,\,]:\fg\otimes\fg\rightarrow \fg$ which preserves the $\Z_2$-grading and satisfies graded versions of the usual Lie bracket axioms. The subspace $\fg_{\0}$ is a Lie algebra under the bracket and $\fg_{\1}$ is a $\fg_{\0}$-module. If $\fg$ has a $\Z$-grading $\fg=\fg_{-1}\oplus\fg_{0}\oplus\fg_{1}$ which is compatible with the $\Z_2$-grading, then $\fg$ is a \emph{type I} Lie superalgebra. If $\fg=\bigoplus_{i=-2}^{i=2}\fg_i$ as a $\Z$-graded Lie superalgebra with $\fg_{\0}=\fg_{-2}\oplus\fg_{0}\oplus \fg_{2}$ and $\fg_{\1}=\fg_{-1}\oplus \fg_{1}$, then it is of \emph{type II}. If there is a connected reductive algebraic group $G_{\0}$ such that $\operatorname{Lie}(G_{\0})=\fg_{\0}$, and an action of $G_{\0}$ on $\fg_{\1}$ which differentiates to the adjoint action of $\fg_{\0}$ on $\fg_{\1}$, then $\fg$ is called \emph{classical}. Moreover, if it has a nondegenerate invariant supersymmetric even bilinear form, then $\fg$ is called \emph{basic}. For our purposes, the Lie superalgebra $\osp(k|2)$ ($k>2$) is a type II basic classical Lie superalgebra (See \cite{Kac}). 

 The category of $\fg$-supermodules is described in \cite{BKN1} but we give a brief definition. Let $U(\fg)$ be the universal enveloping superalgebra. The objects in the aforementioned category are all left $U(\fg)$-modules which are $\Z_2$-graded vector spaces $M=M_{\0}\oplus M_{\1}$ satisfying $U(\fg)_rM_s\subseteq M_{r+s}$ for all $r,\,s\in \Z_2$. For $\fg=\osp(k|2)$ ($k>2$), we only consider the full subcategory, $\F$ (as in \cite{BKN2,BKN3}), of all finite-dimensional $\fg$-supermodules. Note that these are completely reducible over $\fg_{\0}$ since $\fg_{\0}$ is semisimple and hence this aligns with the work done in \cite{BKN1} and \cite{H}. For simplicity, we will from now on use the term ``module" with the understanding that the prefix ``super'' is implicit.
 
\subsection{Complexity} (See \cite[Section~2.2]{BKN1}.) \label{ss:comp}\\
Let $\rv=\{V_t\,|\, t \in \N\} $ be a sequence of finite-dimensional $\C$-vector spaces. The rate of growth of $\rv$, $r(\rv)$, is the smallest nonnegative integer c such that there exists a constant $C > 0$ with $\dim V_t \leq C \cdot t^{c-1}$ for all $t$. If no such integer $c$ exists, then $\rv$ is said to have an infinite rate of growth.

Let $M\in \F$ and $\p\twoheadrightarrow M$ be a minimal projective resolution of $M$. The \emph{complexity} of $M$ is defined to be $c_{\F}(M):=r(\p)$. It was shown in \cite[Theorem~2.5.1]{BKN2} that the complexity is always finite, in particular, $c_{\F}(M)\leq \dim \fg_{\1}$. In addition, one can use the rate of growth of extension groups in $\F$  to compute the complexity (See \cite[Proposition~2.8.1]{BKN2}):
\begin{equation}\label{eq:comp}
c_{\F}(M)=r\Big(\Ext_{(\fg,\fg_{\0})}^\bullet(M,\bigoplus S^{\dim P(S)})\Big),
\end{equation}
where the sum is over all the simple modules $S\in \F$, and $P(S)$ is the projective cover of $S$. Here and elsewhere, we write $\Ext_{(\fg,\fg_{\0})}^\bullet(M,N)$ for the relative cohomology for the pair $(\fg,\fg_{\0})$ as introduced in \cite[Section~2.3]{BKN3}. In this paper, this characterization of the complexity will not be used in computations. However, it shows that the complexity of a module is not a categorical invariant. 

\subsection{$z$-complexity} (See \cite[Section~9]{BKN1}.) \label{zcompdef}\\
The above problem of invariance can be fixed by removing the dimensions of the projective covers from \eqref{eq:comp}. This gives the \emph{$z$-complexity}  of $M\in \F$: 
\begin{equation}\label{eq:zcomp}
z_{\F}(M):=r\Big(\Ext_{(\fg,\fg_{\0})}^\bullet(M,\bigoplus S)\Big),
\end{equation}
where the direct sum runs over all simple modules $S\in\F$. This shows that, unlike complexity, $z_{\F}(-)$ has the advantage of being invariant under category equivalences.

Once again, \eqref{eq:zcomp} will not be used for the purpose of computing the $z$-complexity. However, if $\p\twoheadrightarrow M$ is a minimal projective resolution of $M$, define $s(\p)$ to be the rate of growth of the number of indecomposable summands at each step in the resolution. Then we can see that $z_{\F}(M)=s(\p)$.

\subsection{Support and associated varieties} (See \cite[Section~6]{BKN3}, \cite[Section~2]{DS}.) \label{suppvar}\\
These varieties will be used to give geometric interpretations of the complexity and the $z$-complexity. We only provide brief definitions. 

Let $R=H^{\bullet}(\fg,\fg_{\0};\C)$ be the cohomology ring of $\fg$ and let $M\in \F$. According to \cite[Theorem~2.7]{BKN3}, $\Ext_{\F}^{\bullet}(M,M)$ is a finitely generated $R$-module. Set $J:=Ann_R(\Ext_{\F}^{\bullet}(M,M))$. The \emph{support variety} of $M$ is defined by
\[\fv_{(\fg,\fg_{\0})}(M):=MaxSpec(R/J).\]

Now let $\X$ be the cone of odd self-commuting elements, that is, $\X=\{x\in \fg_{\1} \mid [x,x]=0\}$. If $M \in \F$, then Duflo and Serganova \cite{DS} defined the \emph{associated variety} of $M$ which is equivalent to:
\[\X_M=\{x\in \X \mid \text{$M$ is not projective as a $U(\langle x \rangle)$-module}\}\cup \{0\},\]
where $U(\langle x \rangle)$ denotes the enveloping algebra of the Lie superalgebra generated by $x$.

\subsection{Lie superalgebra $\osp(k|2)$}\label{osp}
We denote by $\fg$ the orthosymplectic Lie superalgebra $\osp(k|2)$ where $k>2$. The matrix form of these Lie superalgebras is given in \cite{Kac}. Note that when $k=2$, $\fg$ is the type I Lie superalgebra $\osp(2|2)$ of type $C$ which was handled in \cite{H}. The case $k=3$ was also handled in \cite{H}. Thus we can assume $k\geq 4$. If $k=2m$,  we have the Lie superalgebra $D(m,1)$ where $\fg_{\0}\cong \so(k)\oplus \s_2$ is of type $D_m\oplus A_1$. If $k=2m+1$ we have the Lie superalgebra $B(m,1)$ where $\fg_{\0}\cong \so(k)\oplus \s_2$ is of type $B_m\oplus A_1$. In both cases, $\fg_{\1}$ is isomorphic, as a $\fg_{\0}$-module, to the outer tensor product of the natural representations of $\so(k)$ and $\s_2$ and thus $\dim \fg_{\1}=2k$. 

The Cartan subalgebra $\fh$ will be chosen to be the set of diagonal matrices.  Let $\fh^*$ denote the dual of the Cartan subalgebra with basis $\{\d, \eps_i\,|\,1\leq i \leq m  \}$. It has a bilinear form defined by
\begin{equation}\label{BLForm}
(\d,\d)=-1,\quad (\delta,\eps_i)=0,\quad (\eps_i,\eps_j)=\d_{i,j}\quad \text{for all $1\leq i,j \leq m$}.
\end{equation}

The set of simple roots is given by:
\begin{equation}\label{Sroots}
\Delta=
\begin{cases}
\{\delta-\eps_1,\eps_{m-1}+\eps_m, \eps_i-\eps_{i+1},\mid 1\leq i\leq m-1\}& \text{if $k=2m$},\\
 \{\delta-\eps_1,\eps_m, \eps_i-\eps_{i+1},\mid 1\leq i\leq m-1\}& \text{if $k=2m+1$}.
           \end{cases}
\end{equation}
Let $\Phi^+$ be the set of positive roots. The set of positive even and odd roots are given respectively:
\begin{equation}\label{PEroots}
\Phi^+_{\0}=\{2\d, \eps_i\pm\eps_j\mid 1\leq i\neq j\leq m\},\, \Phi^+_{\1}=\{\delta\pm \eps_i\mid 1\leq i \leq m\}\quad \text{if $k=2m$},
\end{equation}
\begin{equation}\label{POroots}
\Phi^+_{\0}=\{2\d,\eps_i, \eps_i\pm\eps_j\mid 1\leq i\neq j\leq m\},\, \Phi^+_{\1}=\{\d,\delta\pm \eps_i\mid 1\leq i \leq m\}\quad \text{if $k=2m+1$}.
\end{equation}

For $\alpha \in \Phi^+$, we let $e_\a$ and $e_{-\a}$ to be respectively the positive and negative root vectors as defined in \cite[Section~2]{SZ}.  Let $\fg_{\pm 2}=\C e_{\pm 2\d}$, $\fg_{\pm 1}$ be the $\C$-span of $\{ e_{\pm \a}\,|\,\a\in \Phi^+_{\1}\}$,
 and $\fg_{0}=\so(k)\oplus \C$ whose set of positive roots is $\Phi^+_0=\Phi^+_{\0}\setminus \{2\d\}$. We then have $\fg=\bigoplus_{i=-2}^{i=2}\fg_i$ as a $\Z_2$-consistent $\Z$-graded Lie superalgebra which shows that $\fg$ is a type II Lie superalgebra. 

A weight $\l=\l_0\d+\sum_{i=1}^m\l_i\eps_i\in \fh^*$ will be written as $(\l_0|\l_1,\ldots, \l_m)$ where $\l_i$ will be called the \emph{$i$-th coordinate of $\l$}. Let $\rho=\rho_{\0}-\rho_{\1}$ where $\rho_{\0}$ is half the sum of positive even roots and $\rho_{\1}$ is half the sum of positive odd roots. Let $s=1$ if $k=2m$ and $s=\frac{1}{2}$ if $k=2m+1$. Then by \cite[equation~2.1]{SZ}, we have:
\begin{equation}\label{Rho}
\rho=(s-m|m-s,\ldots,1-s),\quad \rho_{\1}=(m+1-s|0,\ldots,0),\quad \rho_{\0}=(1|m-s,\ldots,1-s).
\end{equation}
In many occasions throughout this paper, we will use the following notation for a $\rho$-translated weight: $$\tilde{\l}=\l+\rho=(\tilde{\l}_0|\tilde{\l}_1,\ldots,\tilde{\l}_m).$$

Let $W_0$ be the Weyl group of $\fg_0$. In particular, $W_0$ is the Weyl group of $\so(m)$ and hence it is a semi-direct product of the symmetric group $S_m$ and $m$-copies of $\Z_2$ which acts on a weight $\l$ by permuting the coordinates and also changing their signs (the number of sign changes must be in $2s\Z_+$). The Weyl group of $\fg$ is $W:=W_0\times \Z_2$ where the nontrivial element of $\Z_2$ changes the sign of the coordinate $\l_0$ when acting on a weight $\l$. The Weyl group acts on $\fh^*$ by the \emph{dot action}: $w\cdot \l=w(\l+\rho)-\rho$ for $w\in W$ and $\l\in \fh^*$. Since $\rho_{\1}$ is $W_0$-invariant, then $w\cdot \l=w(\l+\rho_{\0})-\rho_{\0}$ for $w\in W_0$. 

Integral dominant weights are described in \cite{SZ}. A weight $\l\in \fh^*$ is:

\begin{itemize}
	\item \emph{integral} if $\l_0\in \Z$ and $\l_i\in \Z$ or $\l_i\in s+\Z$, 
\item \emph{regular} if it is integral and  $|\tilde{\l}_1|,\ldots, |\tilde{\l}_m|$ are distinct,
\item \emph{integral $\fg_0$-dominant} if it is integral and: $\l_1\geq \cdots \geq \l_{m-1}\geq |\l_m| \text{ and further, } \l_m\geq 0\,\,if\,\,k=2m+1,$
\item \emph{integral $\fg$-dominant} if it is integral $\fg_0$-dominant and: $l=\l_0 \in \Z_+,\text{ and if } 0\leq l\leq m-1,\text{ then } \l_{l+1}=\l_{l+2}=\cdots=\l_{m}=0.$
\end{itemize}
Denote by $P^{0+}$ and $P^+$ the sets of integral $\fg_0$-dominant and $\fg$-dominant weights respectively. For $\l \in P^{0+}$, the finite-dimensional irreducible $\fg_0$-module of highest weight $\l$ will be denoted by $L^0_\l$. We can then extend it to a $\fg_0\oplus \fg_1\oplus \fg_2$-module by making $(\fg_1\oplus \fg_2)L^0_\l=0$. Then the generalised Verma module $V_\l$ (as defined in \cite{SZ}) is the induced module 
\[V_\l=\operatorname{Ind}_{\fg_0\oplus \fg_1\oplus \fg_2}^\fg L^0_\l\cong U(\fg_{-1}\oplus \fg_{-2})\otimes_\C L^0_\l.\]
Denote by $L(\l) $ the unique irreducible quotient module of $V_\l $. It was shown in \cite{Kac} that $\dim L(\l)< \infty $ if and  only if $\l \in P^+$, thus $P^+$ indexes all simple modules in $\F$.

\subsection{Atypicality}\label{ss:atyp}
One of the important tools in studying representations of Lie superalgebras is atypicality. For a full definition, we refer the reader to \cite[Subsection~7.2]{BKN3}. However, we give the definition in the case of $\osp(k|2)$. A weight $\l\in P^{0+}$ is \emph{atypical} with atypical root $\gamma=\d\pm \eps_l$  if $\tilde{\l}_0=\pm\tilde{\l}_l$ for some $1\leq l\leq m$. In this case, we define the \emph{atypicality} of $\l $, $\atyp(\l )$, to be one. Otherwise, it is called \emph{typical} with $\atyp(\l )=0$. 

We say a simple $\fg$-module $L(\lambda)$ of highest weight $\lambda$ is atypical (typical) if $\l$ is atypical (typical). Thus we define $\atyp(L(\lambda)):=\atyp(\lambda)$. It is known that the atypicality is the same for all simple modules in a given block. Hence it makes sense to refer to the atypicality of a block. 

Moreover, if $P(\l)$ is the projective cover of $L(\l)$, then by \cite[Theorem~1]{Kac1} we know that $P(\lambda)=L(\lambda)$  if $\atyp(\lambda)=0$,  hence $L(\lambda)$ has zero complexity and $z$-complexity. This means that we only need to compute the complexity and $z$-complexity of atypical simple modules. 

In the following, we make two important remarks on atypical weights:
\begin{rk}
In case $\tilde{\l}_0=\tilde{\l}_l=0$ (this can only happen when $k=2m$), both roots $\gamma_{\pm}=\d \pm \eps_l$ are atypical roots of $\l $. In this case, we always choose $\gamma=\d -\eps_l$. we also note that if $\l \in P^{0+}$ is atypical then $\l_i\in \Z$ for all $i$.
\end{rk}
For atypical $\l \in P^{0+}$, define the \emph{atypical type} of $\l $ by 
$\bar{\l }=(|\tilde{\l}_1|,\ldots,|\tilde{\l}_{l-1}|,|\tilde{\l}_{l+1}|,\ldots,|\tilde{\l}_m|) $. We will also make use of the following set $S(\bar{\l})=\{|\tilde{\l}_i|\,:\, i \neq 0,l\}$.

\subsection{Tools}\label{Tools}

The description of the projective covers of the simple modules given in \cite{Coul} utilizes the notation developed in \cite{SZ}. In the following, we recall how this notation is developed. First, we note that every regular weight $\l\in \fh^*$ is $W_0$-conjugate under the dot action to a unique integral $\fg_0$-dominant weight, which will be denoted by $\l^+$. For a weight $\l \in P^{0+}$ with atypical type $\bar{\l}$ and atypical root $\gamma=\d \pm \eps_l$, the authors in \cite{SZ} defined $\l^{\widehat{}}$ and $\l^{\widecheck{}}$ using the following procedure. The first step is to find the smallest positive integers $a_{\pm}$ such that $|\tilde{\l}_l\pm a_{\pm}|\notin S(\bar{\l})$ if $\gamma=\delta +\eps_l$ and $|\tilde{\l}_l\mp a_{\pm}|\notin S(\bar{\l})$ if $\gamma=\delta -\eps_l$. Then 
\begin{equation}
\l^{\widehat{}}=(\l+a_+\gamma)^+\quad \text{and}\quad  \l{\widecheck{}}=(\l-a_-\gamma)^+.
\end{equation}
A couple of examples are given in \cite{SZ} to show how to compute $\l^{\widehat{}}$ and $\l^{\widecheck{}}$. In the following, a few more notations are recalled from \cite{SZ}. We assume that $\bar{\l }$ is a fixed atypical type such that its coordinates $\tilde{\l}_1,\ldots, \tilde{\l}_m$ satisfy $|\tilde{\l}_1|>\cdots > |\tilde{\l}_{m-1}|\geq 0$.
\begin{dfn} (See \cite[Definition~2.9]{SZ})\label{def:bigtilde}
	
Let $j$ be the smallest non-negative integer such that $a:=j+1-s\notin S(\bar{\l})$. Define $\l^{(0)}$ by
\[\widetilde{\l^{(0)}}=\l^{(0)}+\rho = (-a|\tilde{\l}_1, \tilde{\l}_2,\ldots, \tilde{\l}_{m-1-j},a,\tilde{\l}_{m+1-j},\ldots, \tilde{\l}_m),\]
\end{dfn}
where we have re-labeled $\tilde{\l}_{m-j},\ldots, \tilde{\l}_{m-1}$ by $\tilde{\l}_{m+1-j},\ldots, \tilde{\l}_{m}$. 

We will also utilize the weights  $\l^{(i)}$ for $i\in \Z$ which are defined as follows:

\begin{dfn}(See \cite[Definition~2.9]{SZ})
	
If $k=2m+1$ or $k=2m$ with $0\in S(\l)$, define $\l^{(i)}$ ($i\geq 1$) inductively by:
\[\l^{(i)}=\big(\l^{(i-1)}\big)^{\widehat{}},\quad \l^{(-i)}=\big(\l^{(1-i)}\big)^{\widecheck{}}.\]
\end{dfn}

We compute these weights for the weight $\l=(0|0,\ldots,0)$:
\begin{lem}\label{Lem:powers}
 Let $\l=(0|0,\ldots,0)$, then $\l^{(0)}=\l$. For $i\geq 1$:
\[\l^{(i)}=(k+i-3|i-1,0,\ldots,0),\quad \l^{(-i)}=(-i|i,0,\ldots,0).\]
\end{lem}

\begin{proof}
Consider the case $k=2m$. The atypical type of $\l$ is $S(\bar{\l})=\{m-2,m-3,\ldots,1,0\}$. To compute $\l^{\widehat{}}$ we use $\gamma=\delta-\eps_1$, $a_+=2m-2$, and the Weyl group element, $\omega$, that changes the signs of the first and the last coordinates ($i=1,m$). Then
\begin{align*}
\l^{\widehat{}}&=\omega\cdot (m-1|-m+1,m-2,\ldots,1,0)-(1-m|m-1,m-2,\ldots, 1,0)\\
&=(2m-2|0,\ldots,0).
\end{align*}
Similarly, we use $a_-=1$ to find $\l^{\widecheck{}}=(-1|1,0,\ldots,0)$. Since $a=j=m-1$ (in definition~\ref{def:bigtilde}), we have $\l^{(0)}=\l$. Thus, $\l^{(1)}=\l^{\widehat{}}$ and  $\l^{(-1)}=\l^{\widecheck{}}$. An inductive approach can prove the result for $i> 1$ by using $a_+=a_-=1$.  When $k=2m+1$, we use $j=m-1$ and $a=m-1/2$ to show that $\l^{(0)}=\l$. To find $\l^{(\pm 1)}$, we use $a_+=2m-1$ and $a_-=1$. For $\l^{(\pm i)}$ ($i> 1$), we use $a_+=a_-=1$.
\end{proof}

\subsection{Blocks} 
It was shown in \cite{GS} that every atypical block in $\F$ for $\osp(k|2)$ is equivalent to an atypical block in $\osp(3|2)$ if $k$ is odd, or $\osp(4|2)$ or $\osp(2|2)$ if $k$ is even. As stated in \cite[Lemma~11.1]{Coul}, if $k=2m+1$ the quiver diagram of the block $\F_\chi$, for $\chi$ an atypical central character, is equal to the Dynkin diagram of type $D_\infty$. However if $k=2m$, the quiver diagram of $\F_\chi$ is either equal to the Dynkin diagram of type $D_\infty$ or of type $A_\infty^\infty$. Moreover, by \cite[Lemma~11.2]{Coul}, if the quiver diagram of  $\F_\chi$ is of type $D_\infty$, the integral dominant weights corresponding to $\F_\chi$ are $\{\l^{(0)},\l^{(1)},\l^{(2)},\l^{(3)},\dots\}$. However, if the quiver diagram of  $\F_\chi$ is of type $A_\infty^\infty$, the integral dominant weights corresponding to $\F_\chi$ are $\{\l_-^{(0)}=\l_+^{(0)},\l_{\pm}^{(1)},\l_{\pm}^{(2)},\dots\}$. These weights are defined in \cite{SZ} and will only be used slightly in this paper and hence we omit their definitions.

In the following we combined \cite[Lemma~11.2]{Coul} and \cite[Proposition~11.3]{Coul} to get a description of the radical structure of the projective covers.
\begin{lem}\label{projcover}
	If $\l$ is in a block $F_\chi$ with a quiver diagram of type $D_\infty$, the projective indecomposable modules $P(\lambda^{(i)})$ ($i=0,1,2$ and $i\geq 3$) have the following respective radical layer structure:

	\[ \begin{array}{cccccccccc}
	L(\l^{(0)}) & & L(\l^{(1)}) & & L(\l^{(2)}) & & L(\l^{(i)}) \\
	L(\l^{(2)}) &\quad & L(\l^{(2)}) &\quad & L(\l^{(0)})\ \ L(\l^{(1)})\ \ L(\l^{(3)}) &\quad & L(\l^{(i-1)})\ \  L(\l^{(i+1)}) & (i\geq3). \\
	L(\l^{(0)}) & & L(\l^{(1)}) & & L(\l^{(2)}) & & L(\l^{(i)})
	\end{array} \]
	
	If $\l$ is in a block $F_\chi$ with a quiver diagram of type $A_\infty^\infty$, the projective indecomposable modules $P(\lambda_{+}^{(0)})$ and $P(\lambda_{\pm}^{(i)})$ ($i\geq 1$) have the following respective radical layer structure:
	
		\[ \begin{array}{cccccccccc}
	 L(\l_+^{(0)}) & & L(\l_{\pm}^{(i)}) \\
	 L(\l_+^{(1)})\ \ L(\l_-^{(1)}) &\quad & L(\l_{\pm}^{(i-1)})\ \  L(\l_{\pm}^{(i+1)}) & (i\geq1). \\
	 L(\l_+^{(0)}) & & L(\l_{\pm}^{(i)})
		\end{array} \]
	
\end{lem}

\section{Computing the complexity}\label{s:computecomplexity}
In this section we compute the complexity of atypical simple $\osp(k|2)$-modules. Recall that typical simple modules over $\osp(k|2)$ have zero complexity since they are projective. The key idea is to compute the complexity of the trivial module and utilize \cite[Theorem~4.1.1]{kujawa} to find the complexity of other atypical simple modules.   To perform these computations for the trivial module, we first find certain bounds on the dimension of the simple $\so(k)$-module $L(r,0,\ldots,0)$ where $r$ is a positive integer (see $P^{0+}$ in \ref{osp}). 
\begin{lem}\label{lem:bddimsimple}
There are positive constants $C$ and $C'$ that depend only on $k$ such that
\[Cr^{k-2}\leq \dim L(r,0,\ldots,0)\leq C'r^{k-2}.\]
\end{lem}
\begin{proof}
Consider the case $k=2m$. Let $\rho_{\so(k)}$ be half the sum of the positive roots $\Phi^+_{\so(k)}$ in $\so(k)$, then $\rho_{\so(k)}=(m-1,m-2,\ldots,2,1)$. By the Weyl-dimension formula (\cite[Section~24.3]{HUM}) we have:
\[\dim L(r,0,\ldots,0)=\displaystyle{\frac{\prod_{j=2}^m(r+j-1)(2m+r-j-1)\prod_{1<i<j\leq m}(j-i)(2m-i-j)}{\prod_{\alpha\in \Phi^+_{\so(k)}}(\rho_{\so(k)},\alpha)}},\]
thus
\[\dim L(r,0,\ldots,0)=C_1\cdot \prod_{j=2}^m(r+j-1)(2m+r-j-1)\]
where $C_1$ is a constant that depends only on $m$ (and hence only on $k$). Then $\dim L(r,0,\ldots,0)$ is a polynomial in $r$ of degree $2m-2$ with a positive leading coefficient. In fact, since $2\leq j\leq m$, we get
\[
C_1 (r+1)^{m-1}(m+r-1)^{m-1} \leq \dim L(r,0,\ldots,0)\leq C_1(m+r-1)^{m-1}(2m+r-3)^{m-1},
\]
which can be used to show 
\[
C\cdot r^{2m-2}\leq \dim L(r,0,\ldots,0)\leq C'\cdot r^{2m-2},
\]
where  $C$ and $C'$ are positive constants depending only on $k$. The proof is similar when $k=2m+1$.
\end{proof}
This will play a key role in finding the complexity of the atypical simple modules over $\osp(k|2)$ when $k>2$: 
\begin{thm}\label{t:comp}
For atypical $\l\in P^+$, $c_{\F}(L(\l))=k+1$.
\end{thm}

\begin{proof}

First, we will find the complexity of the trivial (atypical) module $\C=L(0|0,\ldots,0)$. The weight $\l=(0|0,\ldots,0)$ is in the block with a quiver diagram of type $D_\infty$. The projective covers in Lemma~\ref{projcover}  have the same structure as those over $\osp(3|2)$ (see \cite[Theorem~2.1.1]{Germoni}), hence we use the same minimal projective resolution of $L(0|0,\ldots,0)$ from \cite[Theorem~6.5.1]{H}. It is given by:
\begin{align}\label{projres1}
 \cdots \rightarrow P_d\rightarrow \cdots \rightarrow P_0=P(\l^{(0)})\rightarrow \C \rightarrow 0,
\end{align}

where the $d$th term,  $d\geq 1$, in this resolution is given by:
\begin{equation}\label{projresterms}
P_d=\begin{cases}
P(\l^{(d+1)})\oplus P(\l^{(d-1)})\oplus \cdots \oplus P(\l^{(2)})\quad \text{if $d$ is odd,}\\
P(\l^{(d+1)})\oplus P(\l^{(d-1)})\oplus \cdots \oplus P(\l^{(3)})\oplus P(\l^{(0)})\quad \text{if $d\equiv 0 \mod 4$,}\\
P(\l^{(d+1)})\oplus P(\l^{(d-1)})\oplus \cdots \oplus P(\l^{(3)})\oplus P(\l^{(1)})\quad \text{if $d\equiv 2 \mod 4$.}
\end{cases}\end{equation}
By using Lemma~\ref{Lem:powers}, a simple $\fg_{\0}$-module with weight $\l^{(i)}$ is of the form: $V(k+i-3)\boxtimes L(i-1,0,\ldots,0)$ where the first module is the $\s_2$-module of dimension $k+i-2$ and the second one is the  $\so(k)$-module whose dimension is bounded by Lemma~\ref{lem:bddimsimple}. Thus for $i\geq 1$, 
\[C(k+i-2)\cdot (i-1)^{k-2}\leq \dim V(k+i-3)\boxtimes L(i-1,0,\ldots,0) \leq C'(k+i-2)\cdot (i-1)^{k-2}.\]

As a $\fg_{\0}$-module, the simple $\fg$-module $L(\mu)$ contains a simple $\fg_{\0}$-module $L_{\0}(\mu)$ as a composition factor. Using the discussion in \cite[Subsection~5.1]{BKN1}, we have the following bounds:
\begin{equation}\label{eq1:Bdprojcover}
\dim L_{\0}(\mu) \leq \dim P(\mu)\leq 2^{\dim \fg_{\1}}\cdot \dim L_{\0}(\mu),
\end{equation}
thus for $i\geq 1$,
\begin{equation}\label{eq1':Bdprojcover}
C(k+i-2)(i-1)^{k-2} \leq \dim P(\l^{(i)})\leq C'\cdot 2^{\dim \fg_{\1}}(k+i-2)(i-1)^{k-2}.
\end{equation}
which can be rewritten as:
\begin{equation}\label{eq1':Bdprojcover}
C_1\cdot i^{k-1} \leq \dim P(\l^{(i)})\leq C_2\cdot i^{k-1},
\end{equation}
for some constants $C_1$ and $C_2$ that depend only on $k$. In each term $P_d$ in \eqref{projresterms}, there are either $\frac{d+1}{2}$, $\frac{d}{2}+1$, or $\frac{d}{2}$ projective covers. Let $T_d$ be the indexing set for these projective modules. Then, for $d\geq 3$, we have $(d+1)^{k-1}\leq d^k$ since $k>2$, and
\[\dim P_d \leq C_2 \sum_{i\in T_d}  \dim P(\l^{(i)}) \leq C_2 \sum_{i=0}^{d-1} d^{k-1}+C_2\cdot (d+1)^{k-1}\leq K_2\cdot d^{k},\]
for some $K_2$ depending only on $k$. This upper bound can be established for $d=1,2$ by direct computations of dimensions. On the other hand, we note that $x=\frac{1}{2}(x+x)\geq \frac{1}{2}(x+x-1)$ and if $t=\frac{d}{2}$ or $\frac{d+1}{2}$, then $t\geq \frac{d}{2}$. Thus
\[\dim P_d \geq C_1 \sum_{i\in T_d}  \dim P(\l^{(i)}) \geq \frac{C_1}{2} \sum_{i=2}^{d+1} i^{k-1} \geq  \frac{C_1}{2} \sum_{i=t}^{d+1} t^{k-1} \geq K_1\cdot d^{k},\]
for some $K_1$ depending only on $k$. 
Hence the complexity of the trivial module is $k+1$. By \cite[Theorem~4.1.1]{kujawa}, all simple modules of the same atypicality have the same complexity. Thus the complexity of all atypical simple $\osp(k|2)$-modules is $k+1$. 
\end{proof}

Next, we give the complexity a geometric interpretation as discussed earlier. We first recall the following result on support varieties:

\begin{prop} \cite[Corollary~4.4.2]{kujawa}\label{suppvarsimple}
	For $\lambda\in P^+$,  $\dim \fv_{(\fg,\fg_{\0})}(L(\lambda))=\atyp(\lambda).$ 
\end{prop}

In the following, we find the dimension of the associated variety of the simple modules:
\begin{prop}\label{assocvarsimple}
For $\lambda\in P^+$, $\dim \X_{L(\lambda)}=\begin{cases}
            k& \text{if $\atyp(\lambda)=1$},\\
            0& \text{if $\atyp(\lambda)=0$.}
            \end{cases}$
\end{prop}

\begin{proof}
When $\lambda$ is typical, $L(\lambda)$ is projective, hence $\dim \X_{X(\lambda)}=0$ by \cite[Theorem~3.4]{DS}. For the rest of the proof, assume $\lambda$ is atypical. \cite[Corollary~2.5]{S} implies that $ \X_{L(\lambda)}=\X$ when $\lambda$ is atypical. Now, $\dim \X$ is equal to the dimension of an irreducible component which can be found by  \cite[Theorem~4.5]{DS}. By computing the dimension given in \cite[Theorem~4.5]{DS} for both cases: $k=2m$ and $k=2m+1$, we found  $\dim \X=k$, thus $\dim \X_{L(\l)}=k$. 
\end{proof}

Combining the computations from Theorem~\ref{t:comp} and Propositions~\ref{suppvarsimple} and \ref{assocvarsimple}, we conclude that:
\begin{thm}\label{ospgeometric1}
For $\lambda\in P^+$, 
$c_{\F}(L(\lambda))=\dim \X_{L(\lambda)}+\dim \fv_{(\fg,\fg_{\0})}(L(\lambda)).$
\end{thm}

\section{Computing the $z$-complexity}\label{s:computeZcomplexity}
We compute the $z$-complexity (cf.\ Subsection~\ref{zcompdef}) of the simple modules over $\osp(k|2)$, $k>2$. If $\l$ is typical, the simple module $L(\l )$ is projective and we can easily show that $z_{\F}(L(\l ))=0$. 

\begin{thm}\label{t:zcompsimple}
If $\l $ is an atypical weight in $P^+$, then $z_{\F}(L(\l ))=2$.
\end{thm}

\begin{proof}
 From Lemma~\ref{projcover} we observed that the projective covers of simple modules in the atypical block with quiver diagram of type $D_\infty$  have the same radical structure as those over $\osp(3|2)$ (\cite[Theorem~2.1.1]{Germoni}). Since the $z$-complexity is a  categorical invariant, we can use  \cite[Theorem~6.7.1]{H} to show that the $z$-complexity of the simple modules in that block is $2$.
 
  On the other hand, the projective covers of simple modules in the atypical block with quiver diagram of type $A^\infty_\infty$  have the same radical structure (a diamond-shape) as those over $\osp(2|2)$. Thus we can use  \cite[Theorem~5.1.1]{H} to show that the $z$-complexity of the simple modules in that block is $2$.
\end{proof}

Next, we give a geometric interpretation of the $z$-complexity using support varieties relative to a detecting Lie subsuperalgebra.  As defined in \cite{BKN3}, let $\ff_{\1}\subseteq \fg_{\1}$ be the span of the root vectors $e_\alpha, f_{\alpha}$ where
$\a=\d - \eps_1$.
Set $\ff_{\0}=[\ff_{\1},\ff_{\1}]$ spanned by $[e_{\a},f_{\a}]$. We define a three-dimensional subalgebra of $\fg$ by
\[\ff:=\ff_{\0}\oplus \ff_{\1}.\]
The Lie superalgebra $\ff$ is classical and so has a support variety theory. Furthermore, as $[\ff_{\0},\ff_{\1}]=0$, it follows that these varieties admit a rank variety description and, in particular, can be identified as subvarieties of $\ff_{\1}$, i.e.,
\begin{align*}
\fv_{(\ff,\ff_{\0})}(M)&=\fv_{\ff_{\1}}^{rank}(M)=\{y\in \ff_{\1}\mid \text{$M$ is not projective as $U(\langle y \rangle)$-module}\}\cup\{0\}.
\end{align*}
Using this detecting subsuperalgebra, we have the following geometric interpretation of the $z$-complexity:
\begin{thm} \label{ospgeometric2}
If $\l \in P^+$, then $\dim \fv_{(\ff,\ff_{\0})}(L(\l ))=z_{\F}(L(\l ))$.
\end{thm}
\begin{proof}
By \cite[Theorem~4.1.1]{kujawa}, any atypical simple module $L(\l )$ satisfies 
$\fv_{(\ff,\ff_{\0})}(L(\l ))=\fv_{(\ff,\ff_{\0})}(\C)$ since $L(\l )$ and $\C$ have the same atypicality. We have, however,  $\fv_{(\ff,\ff_{\0})}(\C)=\ff_{\1}$ which gives $\dim  \fv_{(\ff,\ff_{\0})}(\C)=2$. Thus, 
	$$\dim \fv_{(\ff,\ff_{\0})}(L(\l ))=z_{\F}(L(\lambda))=2.$$ 
	Now, if $L(\l )$ is typical, then it is projective and hence,
	$\dim \fv_{(\ff,\ff_{\0})}(L(\l ))=z_{\F}(L(\l ))=0.$
\end{proof}

\subsection{Acknowledgments} 
I would like to acknowledge Jonathan Kujawa for his feedback on this work. I would also like to thank the referee for his/her valuable comments.


\begin{thebibliography}{99}
\bibitem{Alp} J.~L.~Alperin, Periodicity in groups, \emph{Illinois J. Math}, \textbf{21}, 1977, no. 4,776-783.

\bibitem{BKN1} B.~D.~Boe, J.~R.~Kujawa, and D.~K.~Nakano, Complexity for modules over the classical Lie superalgebra $\mathfrak{gl}(m|n)$, \emph{Compositio Mathematica}, \textbf{148}, 2012, 1561-1592.
\bibitem{BKN2}------, Complexity and module varieties for classical Lie superalgebras, \emph{Int. Math. Res. Not.}, 2011, Issue 3, 696-724.
\bibitem{BKN3} ------, Cohomology and support varieties for Lie superalgebras, Trans. Amer. Math. Soc. \textbf{362}, 2010, no. 12, 6551-6590.
\bibitem{Car}  J.~Carlson, The varieties and the cohomology ring of a module, \emph{Journal of Algebra}, 85, no.\ 1 (1985), 104-43.
\bibitem{Coul} K.~Coulembier, Bott-Borel-Weil theory, BGG reciprocity and twisting functions for Lie superalgebras, \url{http://arxiv.org/pdf/1404.1416v4.pdf}.
\bibitem{DS} M.~Duflo and V.~Serganova, On associated variety for Lie superalgebras, 2005, arXiv:math/0507198.
\bibitem{Germoni} J.~Germoni, Indecomposable representations of $\osp(3,2)$, $\dd$ and $G(3)$, \emph{Colloquium on Homology and Representation Theory} (Spanish) (Vaqueras, 1998) \emph{Bol.  Acad. Nac. Cienc.}, Crdoba, Argentina \textbf{65}, 2000, 147-163.
\bibitem{GS} C.~Gruson and V.~Serganova, Cohomology of generalized supergrassmannians and character formulae for basic classical Lie superalgebras, \emph{Proc. Lond. Math. Soc.} (3) 101, 2010, no. 3, 852-892.
\bibitem{H} H.~El~Turkey, Complexity of modules over classical Lie superalgebras, \emph{Journal of Algebra}, \textbf{445} (2016),  365--393.
\bibitem{HUM} J.~E.~Humphreys, Introduction to Lie algebras and representation theory, \emph{Springer} (1972).
\bibitem{Kac} V.~Kac, Lie superalgebras, \emph{Adv.\  Math.} \textbf{26}, 1977, 8-96.
\bibitem{Kac1} V.~Kac, Representations of classical Lie superalgebras, \emph{Lecture Notes in Math.}, \textbf{676}, Springer Verlag, 1977, 597-626.
\bibitem{kujawa} J.~Kujawa, The generalized Kac-Wakimoto conjecture and support varieties for the Lie superalgebra $\osp(2|2n)$, \emph{SE Lie Theory Conf. Proc., Proc. Symp. Pure Math.}, Volume 86, 2012.
\bibitem{S} V.~Serganova, On superdimension of an irreducible representation of a basic classical Lie superalgebra, \emph{Supersymmetry in Mathematics and Physics,} Springer, Heidelberg, 2011, 253-273.
\bibitem{SZ} Y.~Su and R.~B.~Zhang, Generalised Verma modules for the orthosymplectic Lie superalgebra $\osp_{K|2}$, \emph{Journal of Algebra}, \textbf{357}, 94-115 (2012).

\end{thebibliography}
\end{document}